\documentclass[a4paper]{amsart}

\usepackage{amsmath, amssymb, latexsym, amsfonts, bbm}
\usepackage[matrix,arrow,curve]{xy}
\usepackage{tikz}
\usetikzlibrary{fit,positioning}

\numberwithin{equation}{section}		
\theoremstyle{plain}
\newtheorem*{thmIntr}{Theorem}{\bf}{\it}

\newtheorem{theorem}[equation]{Theorem}

\newtheorem{corollary}[equation]{Corollary}

\newtheorem{lemma}[equation]{Lemma}

\theoremstyle{definition}

\newtheorem{definition}[equation]{Definition}

\theoremstyle{remark}
\newtheorem{remark}[equation]{Remark}

\newcommand{\pt}{{\rm pt}}
\newcommand{\A}{\mathbb{A}}
\newcommand{\SSp}{\mathbb{S}}
\newcommand{\Proj}{\mathbb{P}}
\newcommand{\rank}{\operatorname{rank}}
\newcommand{\triv}{\mathbbm{1}}
\newcommand{\struct}{\mathcal{O}}
\newcommand{\Gm}{{\mathbb{G}_m}}
\newcommand{\SH}{\mathcal{SH}}
\newcommand{\Hp}{\mathcal{H}_\bullet}
\newcommand{\Spec}{\operatorname{Spec}}
\newcommand{\id}{\operatorname{id}}
\newcommand{\codim}{\operatorname{codim}}

\newcommand{\SL}{{\mathrm{SL}}}

\newcommand{\Hom}{{\mathrm{Hom}}}

\newcommand{\W}{\mathrm{W}}
\newcommand{\Th}{\operatorname{Th}}
\newcommand{\thc}{\operatorname{th}}

\begin{document}

\title[Push-forwards for theories with invertible stable Hopf element]{On the push-forwards for motivic cohomology theories with invertible stable Hopf element}

\author{Alexey Ananyevskiy}

\address{Chebyshev Laboratory, St. Petersburg State University, 14th Line, 29b, Saint Petersburg, 199178 Russia}

\email{alseang@gmail.com}

\maketitle

\begin{abstract}
We present a geometric construction of push-forward maps along projective morphisms for cohomology theories representable in the stable motivic homotopy category assuming that the element corresponding to the stable Hopf map is inverted in the coefficient ring of the theory. The construction is parallel to the one given by A.~Nenashev for derived Witt groups. Along the way we introduce cohomology groups twisted by a formal difference of vector bundles as cohomology groups of a certain Thom space and compute twisted cohomology groups of projective spaces. 
\end{abstract}



\section{Introduction}

Existence of push-forward maps in a cohomology theory gives a powerful tool that allows to perform various computations and analyze properties of the considered cohomology theory. The best understood algebraic cohomology theories, such as etale cohomology, Chow groups and algebraic $K$-theory, have push-forward maps for arbitrary projective morphisms. Roughly speaking, oriented cohomology theories (see \cite{LM07,PS03,S07a}) are cohomology theories possessing push-forwards along arbitrary projective morphisms and satisfying certain natural properties. The theory in the oriented setting is quite well-developed: one may obtain a projective bundle theorem and introduce Chern classes of vector bundles \cite{PS09,S07a}, study morphisms between such theories and obtain Riemann-Roch type theorems \cite{PS04,S07b}, construct a universal oriented cohomology theory \cite{LM07} that allows to perform computations in the universal setting, study the corresponding categories of motives and obtain various motivic decompositions \cite{NZ06}, etc.

On the other hand, there are some interesting cohomology theories for which one can not define push-forward maps along arbitrary projective morphisms. Among the examples are derived Witt groups, hermitian K-theory, oriented Chow groups and Witt cohomology (see \cite{Bal99,BM00,Sch10} for the definitions). A more down-to-earth example is given by choosing an embedding of the base field to $\mathbb{R}$, taking real points of the considered variety and computing singular cohomology with integer coefficients. Another example is given by motivic stable cohomotopy groups $\SSp^{i,j}$, i.e. by the cohomology theory represented by the spherical spectrum in the motivic stable homotopy category. All these theories have in common that the usual version of projective bundle theorem fails, i.e. $A(\mathbb{P}^n)\not\cong A(\pt)[t]/t^{n+1}$, where $A$ denotes the corresponding cohomology theory. Nevertheless, sometimes one can obtain a certain computation for projective space, for example for derived Witt groups we know that $\W^*(\Proj^n_k)\oplus \W^*(\Proj^n_k,\struct(1))$ is a free module over $\W^{*}(\Spec k)$ of rank two \cite{G01,Ne09b,Wa03}. Based on this computation A.~Nenashev defined for derived Witt groups push-forwards along projective morphisms \cite{Ne09a}: for a projective morphism of smooth varieties $f\colon Y\to X$ and a line bundle $L$ over $X$ he defined a homomorphism 
\[
f_\W\colon \W^*(Y,f^*L\otimes\omega_Y\otimes f^*\omega_X^{-1})\to \W^{*+c}(X,L),
\]
where $c=\dim X-\dim Y$. The twists should agree in the way as it is stated above, for example, for a projection $p\colon \mathbb{P}^2_k\to \Spec k$ we do not have a push-forward map $p_\W\colon \W^*(\mathbb{P}^2_k)\to \W^{*-2}(\Spec k)$. It is noteworthy that there is another way to define push-forward maps for derived Witt groups based on Grothendieck duality \cite{CH11} yielding homomorphisms of the same kind.

For a cohomology theory representable in the motivic stable homotopy category there is a general approach to the construction of push-forward maps based on the Atiyah duality, which was settled in the motivic setting by P.~Hu and I.~Kriz \cite{H05} via geometric methods and by J.~Riou \cite{R05} using four functors formalism developed by V.~Voevodsky and J.~Ayoub. Consider a projective morphism $f\colon Y\to X$ of smooth varieties. Suppose for simplicity that both $Y$ and $X$ are projective. Then we have the dual morphism $f^\vee\colon \Sigma_T^\infty X^\vee\to \Sigma_T^\infty Y^\vee$ for $\Sigma_T^\infty X^\vee=\operatorname{hom}(\Sigma_T^\infty X,\SSp)$ and $\Sigma_T^\infty Y^\vee = \operatorname{hom}(\Sigma_T^\infty Y,\SSp)$ being the dual spectra. Atiyah duality gives isomorphisms $\Sigma_T^\infty X^\vee\cong \Sigma_T^\infty \Th(-T_X)$ and $\Sigma_T^\infty Y^\vee\cong \Sigma_T^\infty \Th(-T_Y)$, where $\Sigma_T^\infty \Th(-T_X)$ and $\Sigma_T^\infty \Th(-T_Y)$ are suspension spectra of stable normal bundles, i.e. we use the Jouanolou device (see \cite{J73}, \cite[\S~4]{We89}) replacing the varieties by affine ones, consider vector bundles complement to the tangent bundles and take an appropriate shifts of the suspension spectra of the respective Thom spaces. Hence we have the corresponding morphism of the cohomology groups
\[
(f^\vee)^A \colon A^{*,*}(\Sigma_T^\infty \Th(-T_Y))\to A^{*,*}(\Sigma_T^\infty \Th(-T_X)).
\]
Identifying $A^{*,*}(\Sigma_T^\infty \Th(-T_X))$ and $A^{*,*}(\Sigma_T^\infty \Th(-T_Y))$ with the cohomology of $X$ and $Y$ via appropriate Thom isomorphisms we obtain push-forward maps. In particular, recall that derived Witt groups of a Thom space coincide with the derived Witt groups of the vector bundle supported on the zero section, thus Thom isomorphisms from \cite{Ne07} give isomorphisms
\begin{gather*}
\W^*(\Sigma_T^\infty \Th(-T_Y))\cong \W^{*+\dim Y}(Y,(\det T_Y)^{-1})= \W^{*+\dim Y}(Y, \omega_Y)\\
\W^*(\Sigma_T^\infty \Th(-T_X))\cong \W^{*+\dim X}(X,(\det T_X)^{-1})= \W^{*+\dim X}(X, \omega_X)
\end{gather*}
and from our viewpoint that is the main geometric reason why we have the twist by $\omega_Y\otimes f^*\omega_X^{-1}$.

In this paper we generalize the construction of projective push-forwards for derived Witt groups given by A.~Nenashev \cite{Ne09a} to the case of representable cohomology theories with invertible stable Hopf element $\eta\in A^{-1,-1}(\pt)$ (see Definition~\ref{d:hopf} for the precise definition of stable Hopf element). Among the examples are Witt cohomology $\mathrm{H}^*(-,\W^*)$ and stable cohomotopy groups with inverted stable Hopf element $\SSp_\eta^{*,*}(-)$. On the other hand, for an oriented cohomology theory $A^{*,*}(-)$ one can show that $A^{*,*}(-)[\eta^{-1}]=0$, see \cite[Corollary~1]{An12} or just combine Theorem~\ref{Th:proj_coh} for $\Proj^2_X$ with the oriented cohomology projective bundle theorem $A^{*,*}(\Proj^{2}_X)=A^{*,*}(X)[t]/(t^3)$. Thus our study of representable cohomology theories with invertible stable Hopf element does not give any new immediate information on the oriented cohomology theories but rather goes in a complementary direction. It is worth mentioning that although inverting stable Hopf element we lose a lot of information, in particular, all the oriented theories become trivial, the resulting theory is quite rich. Rationally the theory of $\eta$-inverted spectra coincides with the theory of homotopy modules over the Witt sheaf (see \cite{ALP15} for a precise statement), while integrally it is even more complicated: a recent work of M.~Andrews and H.~Miller \cite{AM15} shows that $\SSp_\eta^{*,*}(\pt)$ over $k=\mathbb{C}$ contains some additional interesting $2$-torsion. Here and later on we denote $\pt=\Spec k$ for the base field $k$.

Analyzing A.~Nenashev's approach to the construction of projective push-forwards one can notice that it is based on the computation of derived Witt groups of projective space, which could be done via induction as in \cite{Ne09b} using the following properties of derived Witt groups:
\begin{enumerate}
\item
Identification $\W^*(\mathbb{P}^2)=\W^*(\pt)$;
\item
Thom isomorphisms $\W^{*+n}_X(E)\cong \W^{*}(X,\det E)$ for a rank $n$ vector bundle $E$ over a smooth variety $X$;
\item
Isomorphisms $\W^*(X,L_1\otimes L_2^{\otimes 2})\cong \W^*(X,L_1)$.
\end{enumerate}
For a general representable cohomology theory $A^{*,*}(-)$ one may rephrase these properties in the following way:
\begin{enumerate}
\item
The stable Hopf element $\eta\in A^{-1,-1}(\pt)$ is invertible (see Definition~\ref{d:hopf} and Remark~\ref{r:hopfcone});
\item
$A^{*,*}(-)$ is $\SL$-oriented in the sense of \cite[Definition 5.1]{PW10b} (see also \cite{An12});
\item
For a line bundle $L$ over a smooth variety $X$ one has $A^{*+2,*+1}_X(L^{\otimes 2})\cong A^{*,*}(X)$ (cf. \cite[Definition 3.3]{PW10c}).
\end{enumerate}
As we show in Theorem~\ref{Th:proj_coh} only the first property is essential:
\begin{thmIntr}
Let $A$ be a commutative ring $T$-spectrum and let $X$ be a smooth variety. Then
\[
A^{*,*}_\eta(\Proj^{2n}_X)\cong A^{*,*}_\eta(X),\quad A^{*,*}_\eta(\Proj^{2n-1}_X)\cong A^{*,*}_\eta(X)\oplus A^{*-4n+1,*-2n}_\eta(X),
\]
where $A^{*,*}_\eta(-)=A^{*,*}(-)[\eta^{-1}]$.
\end{thmIntr}
In order to obtain this theorem we consider the projection
\[
H_{2n}\colon (\A^{2n}-\{0\},(1,1,\ldots,1))\to \Proj^{2n-1}/\Proj^{2n-2}
\]
given by $H_{2n}(x_1,x_2,\ldots,x_{2n})=[x_1:x_2:\ldots:x_{2n}]$. It turns out that this projection is, up to canonical isomorphisms, a suspension of the Hopf map. Hence it induces an isomorphism on cohomology groups with inverted $\eta$. This isomorphism allows us to compute cohomology groups of projective spaces inductively. Note that real points of $H_{2n}$ give a morphism $S^{2n-1}\to S^{2n-1}$ and one can easily see that it corresponds to $2\in\pi_{2n-1}(S^{2n-1})$, while real points of the Hopf map give a two-folded covering $S^1\to S^1$.

In order to define push-forwards we adopt ideas arising from Atiyah duality and introduce cohomology groups twisted by a vector bundle as shifted cohomology of the Thom space of the vector bundle. Then, using Jouanolou's device, we define cohomology groups of a smooth variety $X$ twisted by a formal difference of vector bundles as shifted cohomology groups of an appropriate Thom space. These twisted groups are denoted by $A^{*,*}(X;E_1\ominus E_2)$. In particular, we have cohomology groups twisted by a complement to the tangent bundle, $A^{*,*}(X;\ominus T_X)$. The latter groups have nice functoriality properties and serve as a model of the cohomology of the stable normal bundle suspension spectra. 

It is well-known that one can define push-forwards along closed embeddings using deformation to the normal bundle (see, for example, \cite{PS03}). In our setting, for a closed embedding $f\colon Y\to X$ of codimension $c$ we obtain a push-forward map
\[
f_A\colon A^{*,*}(Y;\ominus T_Y)\to A^{*+2c,*+c}(X;\ominus T_X).
\]
This map coincides with $(f^\vee)^A$ described above up to some choices arising from the definition of twisted cohomology groups. In Theorem~\ref{t:twistedpbundlevect} we compute cohomology of $\mathbb{P}^n$ twisted by an arbitrary vector bundle assuming that the stable Hopf element is inverted. This result formally yields the computation for a twist by a formal difference of vector bundles, see Theorem~\ref{t:twistedpbundle}. In particular, we have the following theorem (for the general statement see loc.cit.).
\begin{thmIntr}
Let $A$ be a commutative ring $T$-spectrum and let $X$ be a smooth variety. Then the push-forward map
\[
A^{*,*}_\eta(X;\ominus T_X) \xrightarrow{i_A} A^{*+4n,*+2n}_\eta(\Proj^{2n}_X;\ominus T_{\Proj^{2n}_X})
\]
is a isomorphism. Here $a$ is a rational point on $\Proj^{2n}$ and $i\colon X\to \Proj^{2n}_X$ is the closed embedding given by $i(x)=(x,a)$. 
\end{thmIntr}
For the canonical projection $p\colon \Proj^{2n}_X\to X$ we define the push-forward map
\[
p_A=i_A^{-1}\colon A^{*+4n,*+2n}_\eta(\Proj^{2n}_X;\ominus T_{\Proj^{2n}_X})\to A^{*,*}_\eta(X;\ominus T_X) 
\]
as the isomorphism inverse to $i_A$. Then for a projective morphism $f$ of codimension $c$ we define push-forward map
\[
f_A\colon A^{*,*}_\eta(Y;\ominus T_Y)\to A^{*+2c,*+c}_\eta(X;\ominus T_X)
\]
representing $f$ as a composition $f=p\circ i$ of a closed embedding and a projection and taking push-forwards along these morphisms. This construction generalizes immediately to twisted cohomology groups (see Definition~\ref{def:projpush}) yielding for a codimension $c$ projective morphism $f\colon Y\to X$ and a formal difference of vector bundles $E_1\ominus E_2$ over $X$ push-forward map
\[
f_A\colon A^{*,*}_\eta (Y; f^*E_1\ominus (f^*E_2\oplus T_Y))\to A^{*+2c,*+c}_\eta (X; E_1\ominus (E_2\oplus T_X)).
\]
Using A.~Nenashev's constructions (which follow in general the ones introduced in \cite{PS09,S07a}) one can check that this definition does not depend on the choice of $p$ and $i$ and obtain the usual properties of push-forwards: functoriality, projection formula and compatibility with transversal base change.

The paper is organized in the following way. In Section~2 we recall some well known facts about motivic homotopy theory and representable cohomology theories. In the next section we introduce the language of the cohomology groups twisted by a vector bundle and check some basic properties, in particular, in Corollary~\ref{c:K0_coh} we show that twisted cohomology groups depend only on the class of the twist in reduced $K_0$. The main part of the paper is Section~4 where we compute twisted cohomology of projective spaces assuming that the stable Hopf element is inverted. Then we generalize the twisted cohomology setting to the case of the formal differences of vector bundles. In the last section we define push-forwards along projective morphisms and sketch its basic properties.

\section{Preliminaries on motivic homotopy theory}

In this section we recall some basic definitions and constructions in the nonstable and stable motivic homotopy categories $\Hp(k)$ and $\SH(k)$. We refer the reader to the foundational papers \cite{MV99,V98} for the details. 

Let $k$ be a field and let $\mathrm{Sm}/k$ be the category of smooth varieties over $k$. 

\begin{definition}
A motivic space over $k$ is a simplicial presheaf on $\mathrm{Sm}/k$. Every smooth variety $X$ defines a motivic space $\Hom_{\mathrm{Sm}/k}(-,X)$ constant in the simplicial direction. We will occasionally write $\pt$ for $\Spec k$ regarded as a motivic space. We use the injective model structure on the category of the pointed motivic spaces $M_\bullet(k)$. Inverting all the weak motivic equivalences in $M_\bullet(k)$ (see \cite{MV99}) we obtain the pointed motivic unstable homotopy category $\Hp(k)$.
\end{definition}

\begin{definition}
For a vector bundle $E$ over a smooth variety $X$ we put $\Th(E)=E/(E-X)\in \Hp(k)$ for \textit{the Thom space of $E$}.
\end{definition}

\begin{definition}
Let $X$ be a smooth variety and let $\triv^n_X$ be a trivialized vector bundle over $X$. An automorphism $\alpha\in \operatorname{Aut}(\triv^n_X)$ is \textit{elementary} if the corresponding matrix $A\in \operatorname{GL}_n(k[X])$ is elementary, i.e. belongs to the subgroup
\[
\operatorname{E}_n(X)= \langle\, \mathrm{I}_n+ g(x)e_{ij}\,|\, g(x)\in k[X],\, i\neq j \,\rangle\le \operatorname{GL}_n(k[X])
\]
generated by transvections (shear mappings). Here $\mathrm{I}_n$ is the identity matrix of size $n\times n$, $g(x)$ is a regular function and $e_{ij}$ is a matrix unit (matrix with $1$ at $(i,j)$ and $0$ everywhere else).

\noindent Typical examples of elementary matrices are given by a determinant $1$ matrix with the coefficients from the base field $k$ and by the matrix $\begin{pmatrix} s(x)^{-1} & 0 \\ 0 & s(x) \end{pmatrix}$ with $s(x)\in k[X]^*$.
\end{definition}

The following lemma is well known.

\begin{lemma} \label{l:elematrix}
Let $E$ be a vector bundle over a smooth variety $X$, $\triv^n_X$ be a trivialized vector bundle over $X$ and $\alpha\in \operatorname{Aut}(\triv^n_X)$ be an elementary automorphism. Then $\alpha\otimes \id\in \operatorname{Aut}(\triv^n_X\otimes E)= \operatorname{Aut}(E^{\oplus n})$ induces trivial automorphisms on $\Th(E^{\oplus n})$ and $(\Proj(E^{\oplus n})_+,+)$ in $\Hp(k)$.
\end{lemma}
\begin{proof}
The proof does not depend on the vector bundle $E$, so we assume that $E=\triv_X$. Moreover, the proofs for the Thom space and for the projective bundle are quite the same, so we give the detailed proof only in the first case.

The matrix $A$ corresponding to $\alpha$ can be represented as a product 
\[
A=\prod_k t_{i_kj_k}(g_k(x)),
\]
where $t_{i_kj_k}(g_k(x))=\mathrm{I}_n + g_k(x)e_{i_kj_k}$. Put
\[
A(t)=\prod_k t_{i_kj_k}(tg_k(x))\in \operatorname{GL}_n(k[X][t]).
\]
Denote $p\colon X\times \mathbb{A}^1\to X$ the projection and let $i_0,i_1\colon X\to X\times\mathbb{A}^1$ be the embeddings given by $i_0(x)=(x,0)$ and $i_1(x)=(x,1)$. These maps induce isomorphisms
\[
\overline{p}\colon \Th(p^*\triv^n_X)\xrightarrow{\simeq} \Th(\triv^n_X),\quad \overline{i}_0,\overline{i}_1\colon \Th(\triv^n_X)\xrightarrow{\simeq} \Th(p^*\triv^n_X).
\]
We have $\overline{p}\overline{i}_0=\overline{p}\overline{i}_1=\id_{\Th(\triv^n_X)}$, hence $\overline{i}_0=\overline{i}_1=\overline{p}^{-1}$ in $\Hp(k)$. On the other hand $A(t)$ defines an automorphism $\overline{A}(t)\colon \Th(p^*\triv^n_X)\xrightarrow{\simeq} \Th(p^*\triv^n_X)$ and 
\[
\id_{\Th(\triv^n_X)}=\overline{p}\overline{A}(t)\overline{i}_0=\overline{p}\overline{A}(t)\overline{i}_1=\overline{\alpha},
\]
where $\overline{\alpha}\colon \Th(\triv^n_X)\xrightarrow{\simeq} \Th(\triv^n_X)$ is the automorphism induced by $\alpha$.
\end{proof}

\begin{definition}
Let $T=\A^1/(\A^1-\{0\})$ be the Morel-Voevodsky object. A $T$-spectrum \cite{Jar00} $M$ is a sequence of pointed motivic spaces $(M_0,M_1,M_2,\dots)$ equipped with structural maps $\sigma_n  \colon T\wedge M_n\to M_{n+1}$. A map of $T$-spectra is a sequence of maps of pointed motivic spaces which is compatible with the structure maps. Inverting the stable motivic weak equivalences as in \cite{Jar00} we obtain the motivic stable homotopy category $\SH(k)$.

A pointed motivic space $Y$ gives rise to a suspension $T$-spectrum $\Sigma^{\infty}_T Y$. Denote $\SSp=\Sigma^{\infty}_T (\pt_+)$ the spherical spectrum. Both $\Hp(k)$ and $\SH(k)$ are equipped with symmetric monoidal structures $(\wedge,\pt_+)$ and $(\wedge,\SSp)$ respectively and $\Sigma^{\infty}_T\colon \Hp(k)\to \SH (k)$ is a strict symmetric monoidal functor. 
\end{definition}

\begin{definition}
There are two spheres in $M_\bullet(k)$, $S^{1,0}=\Delta^1/\partial(\Delta^1)$ and $S^{1,1}=(\Gm,1)$. Here we follow the notation and indexing introduced in \cite[p.111]{MV99}. For the integers $p,q\ge 0$ we write $S^{p+q,q}$ for $(S^{1,0})^{\wedge p} \wedge (S^{1,1})^{\wedge q} $ and $\Sigma^{p+q,q}$ for the suspension functor $-\wedge S^{p+q,q}$. This functor becomes invertible in the stable homotopy category $\SH(k)$, so we extend the notation to arbitrary integers $p,q$ in an obvious way. 
\end{definition}

\begin{definition}
Any $T$-spectrum $A$ defines a bigraded cohomology theory on the category of pointed motivic spaces. Namely, for a pointed motivic space $Y$ one sets
\[
A^{p,q}(Y)=\Hom_{\SH(k)}(\Sigma^\infty_T Y,\Sigma^{p,q}A)
\]
and $A^{*,*}(Y)=\bigoplus_{p,q}A^{p,q}(Y)$. We can regard a smooth variety $X$ as an externally pointed motivic space $(X_+,+)$. Put $A^{p,q}(X)=A^{p,q}(X_+,+)$. In case of $i-j,j\ge 0$ there is a suspension isomorphism $A^{p,q}(Y)\cong A^{p+i,q+j}(\Sigma^{i,j}Y)$ induced by the shuffling $S^{p,q}\wedge S^{i,j}\cong S^{p+i,q+j}$. In the motivic homotopy category there is a canonical isomorphism $T\cong S^{2,1}$ \cite[Lemma 2.15]{MV99}, we write $\Sigma_T\colon A^{*,*}(Y)\xrightarrow{\simeq} A^{*+2,*+1}(Y\wedge T)$ for the corresponding suspension isomorphism.
\end{definition}

\begin{definition} \label{d:coht}
A {\it commutative ring $T$-spectrum} is a commutative monoid $A$ in $(\SH(k),\wedge,\SSp)$. We recall some properties of the cohomology theory represented by a commutative ring $T$-spectrum $A$ that we are going to use in the paper.

\begin{enumerate}

\item \textit{Localization:} let $i\colon Z\to X$ be a closed embedding of varieties with smooth $X$. Denote $j\colon U=X-Z\to X$ the open complement and $z\colon X\to X/U$ the quotient map. Then there is a long exact sequence
\[
\dots \xrightarrow{\partial} A^{*,*}(X/U)\xrightarrow{z^A} A^{*,*}(X)\xrightarrow{j^A} A^{*,*}(U)\xrightarrow{\partial} A^{*+1,*}(X/U)\xrightarrow{z^A}\dots
\]
More generally, let $j\colon W\to Y$ be a cofibration of motivic spaces, denote $z\colon Y\to Y/W$ the quotient map. Then there is a long exact sequence
\[
\dots \xrightarrow{\partial} A^{*,*}(Y/W)\xrightarrow{z^A} A^{*,*}(Y)\xrightarrow{j^A} A^{*,*}(W)\xrightarrow{\partial} A^{*+1,*}(Y/W)\xrightarrow{z^A}\dots
\]

\item \textit{Homotopy invariance:} for an $\A^n$-bundle $p\colon E\to X$ over a smooth variety $X$ the induced homomorphism $p^A\colon A^{*,*}(X)\to A^{*,*}(E)$ is an isomorphism.

\item \textit{Mayer-Vietoris:} if $X=U_1\cup U_2$ is a union of two open subsets $U_1$ and $U_2$ then there is a natural long exact sequence
\[
\dots \to A^{*,*}(X)\to A^{*,*}(U_1)\oplus A^{*,*}(U_2)\to A^{*,*}(U_1\cap U_2)\to A^{*+1,*}(X)\to \dots
\]

\item \textit{Cup-product:} for a pointed motivic space $Y$ we have a functorial graded ring structure
\[
\cup\colon A^{*,*}(Y)\times A^{*,*}(Y)\to A^{*,*}(Y).
\]
Moreover, let $i_1\colon Z_1\to X$ and $i_2\colon Z_2\to X$ be closed embeddings of varieties with $X$ being smooth. Then we have a functorial, bilinear and associative cup-product
\[
\cup \colon A^{*,*}(X/(X-Z_1))\times A^{*,*}(X/(X-Z_2)) \to A^{*,*}(X/(X-Z_1\cap Z_2)).
\]
In particular, setting $Z_1=X$ we obtain an $A^{*,*}(X)$-module structure on $A^{*,*}(X/(X-Z_2))$. All the morphisms in the localization sequence are homomorphisms of $A^{*,*}(X)$-modules.

\item \textit{Module structure over stable cohomotopy groups:} for every motivic space $Y$ the $\wedge$--product in $\SH(k)$ defines an $\SSp^{*,*}(\pt)$-module structure on $A^{*,*}(Y)$. For a smooth variety $X$ the ring $A^{*,*}(X)$ is a graded $\SSp^{*,*}(\pt)$--algebra with the unit $1$ being the image of the identity map under the composition $\SSp^{*,*}(\pt)\to \SSp^{*,*}(X)\to A^{*,*}(X)$. Here the first morphism is induced by the projection $X\to \pt$ and the second one is given by the unit morphism $\SSp\to A$.

\item \textit{Push-forwards for closed embeddings:} let $i\colon Y\to X$ be a closed embedding of smooth varieties with a rank $n$ normal bundle $N_i$ and let $E$ be a vector bundle over $X$. Denote $i^E\colon Y\to E$ the composition of $i$ and the zero section of $E$. Note that $N_{i^E} \cong i^*E \oplus N_i$. The homotopy purity property \cite[Section 3, Theorem 2.23]{MV99} gives a canonical isomorphism 
\[
d_{i^E}\colon E/(E-Y)\xrightarrow{\simeq} \Th(N_{i^E}) \cong \Th(i^*E\oplus N_i).
\]
Let $z\colon E/(E-X)\to E/(E-Y)$ be the quotient map. Then 
\[
i^E_A=z^A\circ d^A_{i^E}\colon A^{*,*}(\Th(i^*E\oplus N_i))\to A^{*,*}(\Th(E))
\]
is the \textit{push-forward map}. We will usually omit vector bundle from the notation for the push-forward map and write $i_A=i^E_A$.

\end{enumerate}
\end{definition}

\begin{definition} \label{d:hopf}
The \textit{Hopf map} is the morphism of pointed motivic spaces
\[
H\colon (\A^2-\{0\},(1,1)) \to (\Proj^1,[1:1]),\quad H(x,y)=[x:y].
\]
Using canonical isomorphisms $(\A^2- 0,(1,1))\cong S^{3,2}$ (see \cite[Example 2.20]{MV99}) and $(\Proj^1,[1:1])\cong S^{2,1}$ (see \cite[Lemma 2.15 and Corollary 2.18]{MV99}) we may regard $H$ as an element of $\Hom_{\Hp(k)}(S^{3,2},S^{2,1})$. The \textit{stable Hopf element} is the unique $\eta \in \SSp^{-1,-1}(\pt)$ satisfying $\Sigma^{3,2}\eta=\Sigma^\infty_T H$, i.e. $\eta$ is the stabilization of $H$ moved to $\SSp^{-1,-1}(\pt)$ via the canonical isomorphisms. For a commutative ring $T$-spectrum $A$ we will usually denote by the same letter $\eta$ the corresponding element $\eta\in A^{-1,-1}(\pt)$ in the coefficient ring.
\end{definition}

\begin{definition}
Let $A$ be a commutative ring $T$-spectrum and let $Y$ be a pointed motivic space. Denote
\[
A^{*,*}_\eta(Y)=A^{*,*}(Y)\otimes_{\SSp^{*,*}(\pt)} ({\SSp^{*,*}(\pt)}[\eta^{-1}]).
\]
One can easily check that $A^{*,*}_\eta(-)$ is a cohomology theory and satisfies properties from Definition~\ref{d:coht}. We refer to $A^{*,*}_\eta(-)$ as \textit{cohomology theory $A$ with inverted stable Hopf element}.
\end{definition}

\begin{remark} \label{r:hopfcone}
It is well known that $\eta\in A^{-1,-1}(\pt)$ is invertible if and only if the morphism $p^A\colon A^{*,*}(\pt)\xrightarrow{\simeq} A^{*,*}(\Proj^2)$ induced by the projection $p\colon \Proj^2\to \pt$ is an isomorphism. This follows from the fact that the mapping cone of the Hopf map is equivalent to $(\Proj^2,[0:0:1])$ \cite[Lemma~6.2.1]{Mor04}.
\end{remark}

\section{Cohomology twisted by a vector bundle}
In this section we discuss the language of \textit{twisted cohomology groups}. For a vector bundle $E$ over $X$ twisted cohomology groups are defined to be the cohomology groups of $E$ supported on $X$ and shifted by an appropriate degree. If the cohomology theory is oriented, then the twisted cohomology groups canonically coincide with the ordinary ones, but in general they are quite different. A handy property of these groups is that they depend only on the class $[E]$ in $\widetilde{K}_0(X)$.

\begin{definition} \label{d:tw}
Let $E$ be a vector bundle of rank $n$ over a smooth variety $X$ and let $A$ be a commutative ring $T$-spectrum. Denote
\[
A^{*,*}(X;E)=A^{*+2n,*+n}(\Th(E))
\]
and refer to it as \textit{cohomology groups of $X$, twisted by $E$}. For a morphism of smooth varieties $f\colon Y\to X$ denote
\[
f^A=f^A_E\colon A^{*,*}(X;E)\to A^{*,*}(Y;f^*E)
\] 
the homomorphism induced by the corresponding map $\Th(f^*E)\to \Th(E)$. For a sequence of morphisms of smooth varieties $Z\xrightarrow{g} Y\xrightarrow{f} X$ one clearly has 
\[
(f\circ g)^A=g^A\circ f^A\colon A^{*,*}(X;E)\to A^{*,*}(Z;(fg)^*E).
\]
For an isomorphism of vector bundles $\theta\colon E\to E'$ denote
\[
\theta^A\colon A^{*,*}(X;E')\to A^{*,*}(X;E)
\] 
the homomorphism induced by the corresponding map $\Th(E)\to \Th(E')$.

Using this notation we rewrite push-forwards for closed embeddings introduced in Definition~\ref{d:coht} as
\[
i_A\colon A^{*,*}(Y;i^*E\oplus N_i)\to A^{*+2n,*+n}(X;E).
\]
\end{definition}

\begin{remark}
This notation is inspired by the next observation which I learned from Ivan Panin who attributed it to Charles Walter. Let $L$ be a line bundle over a smooth variety $X$. Then for derived Witt groups introduced by Paul Balmer \cite{Bal99} one has a canonical isomorphism $\W^{*+1}_X(L)\cong \W^*(X;L)$ of derived Witt groups with support and derived Witt groups with the twisted duality \cite[Theorem 2.5]{Ne07}. Thus for a general cohomology theory $A$ and a line bundle $L$ over a smooth variety $X$ one can introduce twisted cohomology groups as cohomology groups with support: $A(X;L)=A_X(L)$. See also \cite[Definition~2.2]{Zib11} where the same notation is used in hermitian $K$-theory and real topological $K$-theory.
\end{remark}

\begin{remark}
\label{r:untwist}
For an oriented cohomology theory \cite{PS03} one has Thom isomorphisms $A^{*,*}(X;E)\cong A^{*,*}(X)$ functorial in $X$. For symplectically and $\SL$--oriented cohomology theories \cite{PW10a,PW10b,An12} one has similar Thom isomorphisms for symplectic and special linear vector bundles respectively.
\end{remark}

\begin{definition} \label{d:cup}
Let $E_1$ and $E_2$ be vector bundles over a smooth variety $X$. The cup-product
\begin{multline*}
A^{*,*}((E_1\oplus E_2)/(E_1\oplus E_2 - E_2))\times A^{*,*}((E_1\oplus E_2)/(E_1\oplus E_2 - E_1)) \xrightarrow{\cup} \\
\xrightarrow{\cup} A^{*,*}((E_1\oplus E_2)/(E_1\oplus E_2 - X))
\end{multline*}
combined with the isomorphisms
\begin{gather*}
A^{*,*}((E_1\oplus E_2)/(E_1\oplus E_2 - E_2))\cong A^{*,*}(E_1/(E_1 - X)),\\
A^{*,*}((E_1\oplus E_2)/(E_1\oplus E_2 - E_1))\cong A^{*,*}(E_2/(E_2 - X))
\end{gather*}
induced by contractions of respective bundles gives a product 
\[
\cup \colon A^{*,*}(X;E_1)\times A^{*,*}(X;E_2) \to A^{*,*}(X;E_1\oplus E_2).
\]
\end{definition}

\begin{lemma}\label{l:dual_coh}
Let $E$ be a vector bundle over a smooth variety $X$. Then for the dual vector bundle $E^\vee$ there is a natural isomorphism of $A^{*,*}(X)$-modules
\[
A^{*,*}(X;E)\cong A^{*,*}(X;E^\vee).
\]
\end{lemma}

\begin{proof}
Denote $n=\rank E$ and let $E^0$ be the complement to the zero section of $E$. Consider variety 
\[
Y=\left\{\left.(v,f)\in E\times_X E^\vee\,\right|\, f(v)=1\right\}.
\]
The projection $Y\to E^0$ identifies $Y$ with an affine bundle over $E^0$ yielding $A^{*,*}(E^0)\cong A^{*,*}(Y)$. The projection $E\times_X E^\vee\to E$ induces a homomorphism
\[
\alpha\colon A^{*-2n,*-n}(X;E)\to A^{*,*}((E\times_X E^\vee)/Y).
\]
Consider the following localization sequences.
\[
\xymatrix @R=1pc @C=1pc{
\dots \ar[r] & A^{*-2n,*-n}(X;E) \ar[r]\ar[d]^\alpha & A^{*,*}(X) \ar[r]\ar[d]^\simeq & A^{*,*}(E^0) \ar[r] \ar[d]^\simeq & \dots \\
\dots \ar[r] & A^{*,*}((E\times_X E^\vee)/Y) \ar[r]  & A^{*,*}(E\times_X E^\vee) \ar[r] & A^{*,*}(Y) \ar[r] & \dots \\
}
\]
The five lemma yields that $\alpha$ is an isomorphism. Applying the same reasoning to $E^\vee$ one obtains an isomorphism $A^{*-2n,*-n}(X;E)\cong A^{*,*}((E\times_X E^\vee)/Y)$ whence the claim.
\end{proof}

\begin{lemma}
\label{l:twisted_cancel}
Let $E_1$ and $E_2$ be vector bundles over a smooth variety $X$. Denote $p\colon E_1-X\to X$ the canonical projection and suppose that there exists some $\thc\in A^{*,*}(X;E_2)$ such that 
\[
A^{*,*}(X)\xrightarrow{-\cup \thc} A^{*,*}(X;E_2),\quad A^{*,*}(E_1-X)\xrightarrow{-\cup p^A (\thc)} A^{*,*}(E_1-X;p^*E_2)
\]
are isomorphisms. Then 
\[
A^{*,*}(X;E_1)\xrightarrow{-\cup \thc} A^{*,*}(X;E_1\oplus E_2)
\]
is an isomorphism as well.
\end{lemma}
\begin{proof}
Consider the localization sequence for the zero section of $E_1$ and its twisted version:
\[
\xymatrix @R=1pc @C=1pc{
\dots \ar[r] & A^{*,*}(X;E_1) \ar[r] \ar[d]_{-\cup \thc} & A^{*,*}(X) \ar[r] \ar[d]_{-\cup \thc}^\simeq & A^{*,*}(E_1-X) \ar[r] \ar[d]_{-\cup p^A(\thc)}^\simeq & \dots \\
\dots \ar[r] & A^{*,*}(X;E_1\oplus E_2) \ar[r] & A^{*,*}(X;E_2) \ar[r] & A^{*,*}(E_1-X;p^*E_2) \ar[r] & \dots 
}
\]
The claim follows via the five lemma.
\end{proof}

\begin{corollary} \label{c:sluntwist}
Let $E$ be a vector bundle over a smooth variety $X$ and $A^{*,*}(-)$ be an $\SL$--oriented cohomology theory in the sense of \cite[Definition 5.1]{PW10b} represented by a commutative ring $T$-spectrum $A$. Then there is a canonical isomorphism
\[
A^{*,*}(X;E)\cong A^{*,*}(X;\det E).
\]
\end{corollary}
\begin{proof}
The canonical trivializations 
\[
 \det (\det E\oplus \det E^\vee)\cong \triv_X,\quad \det (E\oplus \det E^\vee)\cong \triv_X,
\]
combined with the above lemma and Remark~\ref{r:untwist} yield the claim:
\[
A^{*,*}(X;E)\cong A^{*,*}(X;E\oplus \det E^{\vee} \oplus \det E)\cong A^{*,*}(X;\det E\oplus \det E^{\vee}\oplus E) \cong A^{*,*}(X;\det E).
\]
\end{proof}

\begin{corollary} \label{c:suspuntwist}
Let $E$ be a vector bundle over a smooth variety $X$ and let $\triv_X^n$ be a trivialized vector bundle over $X$. Then the morphism
\[
A^{*,*}(X;E) \xrightarrow{-\cup\Sigma_T^n 1} A^{*,*}(X;E\oplus \triv_X^n)
\]
is an isomorphism.
\end{corollary}

\begin{lemma}
\label{l:twq}
Let $E$ be a vector bundle over a smooth variety $X$ and let 
\[
0\to E_1\xrightarrow{i} E_2\xrightarrow{p} E_3\to 0
\]
be an exact sequence of vector bundles over $X$. Then there is a canonical isomorphism
\[
A^{*,*}(X;E\oplus E_2)\cong A^{*,*}(X; E\oplus E_3\oplus E_1).
\]
\end{lemma}
\begin{proof}
Using the Jouanolou's device (see \cite{J73}, \cite[\S 4]{We89}) one may assume that $X$ is affine. Then the short exact sequence splits producing a non-canonical isomorphism $\phi\colon E_3\oplus E_1\xrightarrow{\simeq} E_2$. It is sufficient to show that isomorphism 
\[
\phi^A\colon A^{*,*}(X;E\oplus E_2)\xrightarrow{\simeq} A^{*,*}(X; E\oplus E_3\oplus E_1)
\]
does not depend on the choice of the splitting. 

Consider two splittings 
\[
\xymatrix{
E_1 \ar@/^0.5pc/[r]^{i} & E_2 \ar@/^0.5pc/[l]^{j_1} \ar@/_0.5pc/[r]_{p} & E_3 \ar@/_0.5pc/[l]_{s_1}
}
\quad\quad
\xymatrix{
E_1 \ar@/^0.5pc/[r]^{i} & E_2 \ar@/^0.5pc/[l]^{j_2} \ar@/_0.5pc/[r]_{p} & E_3 \ar@/_0.5pc/[l]_{s_2}
}.
\]
They induce isomorphisms 
\[
\phi_1=(s_1,i), \phi_2=(s_2,i)\colon E_3\oplus E_1\xrightarrow{\simeq} E_2.
\]
The inverse isomorphisms are given by $\begin{pmatrix} p \\ j_1 \end{pmatrix}$ and $\begin{pmatrix} p \\ j_2 \end{pmatrix}$ respectively. We need to check that 
\[
\phi_2^{-1}\phi_1=\begin{pmatrix} p\circ s_1 & p\circ i \\ j_2\circ s_1 & j_2\circ i \end{pmatrix}= \begin{pmatrix} \id_{E_3} & 0 \\ j_2\circ s_1 & \id_{E_1} \end{pmatrix}
\]
induces the identity automorphism on the twisted cohomology groups. We claim that this map induces the trivial automorphism of $\Th(E_3\oplus E_1)$ in $\Hp(k)$. This can be shown in the similar way as in the proof of Lemma~\ref{l:elematrix}: inserting $t$ in the lower-left entry of the matrix one obtains an explicit $\A^1$-homotopy.
\end{proof}

\begin{corollary}\label{c:K0_coh}
Let $E_1$ and $E_2$ be vector bundles over a smooth variety $X$ such that $[E_1]=[E_2]$ in $\widetilde{K}_0(X)=K_0(X)/(\mathbb{Z}\cdot [\triv_X])$. Then there is a (non-canonical) isomorphism
\[
A^{*,*}(X;E_1)\cong A^{*,*}(X;E_2).
\]
\end{corollary}
\begin{proof}
Follows from the above lemma and Corollary~\ref{c:suspuntwist}.
\end{proof}

\section{Twisted cohomology of projective spaces}

This section is the main part of the paper, here we compute groups $A^{*,*}_\eta(\mathbb{P}^n_X; p_1^*V\oplus p_2^*E)$, where $E$ is a vector bundle over $\mathbb{P}^n$, $V$ is a vector bundle over $X$, $p_1\colon \mathbb{P}^n_X\to X$ and $p_2\colon \mathbb{P}^n_X\to \mathbb{P}^n$ are canonical projections. In the next two sections this computation will allow us to define projective push-forwards for $A^{*,*}_\eta(-)$.

We start with the following easy but useful lemma.
\begin{lemma}
\label{lemma_square_symbol}
Let $X$ be a smooth variety, $s\in k[X]^*$ be an invertible regular function and $f_{s^2}\colon X_+\wedge T\to X_+\wedge T$ be the morphism given by $f_{s^2}(x,u)=(x,s(x)^2u)$. Then $\Sigma_T^\infty f_{s^2}=\id$ in $\SH(k)$.
\end{lemma}
\begin{proof}
The canonical isomorphism $T\cong (\Proj^1,\infty)$ combined with the splitting 
\begin{multline*}
\Hom_{\SH(k)}((X\times\Proj^1)_+, X_+\wedge (\Proj^1,\infty))=\\
= \Hom_{\SH(k)}(X_+\wedge (\Proj^1,\infty), X_+\wedge (\Proj^1,\infty))\oplus \Hom_{\SH(k)}(X_+, X_+\wedge (\Proj^1,\infty))
\end{multline*}
yields that it is sufficient to show that
\[
\widetilde{f}_{s^2}=g\in \Hom_{\Hp(k)}((X\times\Proj^1)_+, X_+\wedge (\Proj^1,\infty)),
\] 
where $\widetilde{f}_{s^2}(x,[u:v])=(x,[s(x)^2u:v])$ and $g(x,[u:v])=(x,[u:v])$. We have 
\[
g(x,[u:v])=(x,[u:v])=(x,[s(x)u:s(x)v]),
\]
hence $g=\widetilde{f}_{s^2}\circ h$ with 
\[
h\colon (X\times\Proj^1)_+\to (X\times\Proj^1)_+,\quad h(x,[u:v])=(x,[s(x)^{-1}u,s(x)v]).
\]
By Lemma~\ref{l:elematrix} we have $h=\id$ in $\Hp(k)$ and the claim follows.
\end{proof}

\begin{lemma}
\label{HopfSuspended}
Let $A$ be a commutative ring $T$-spectrum and $X$ be a smooth variety. Denote 
\[
0_\Proj=[0:0:\ldots:0:1]\in\mathbb{P}^{2n-1},\quad 1_\A=(1,1,\ldots, 1)\in (\A^{2n}-0).
\]
Then the projection $H_{2n}\colon (\A^{2n}-0)\to \Proj^{2n-1}$ given by 
\[
H_{2n}(x_1,x_2,\ldots,x_{2n})=[x_1:x_2:\ldots x_{2n}]
\]
induces an isomorphism 
\[
\overline{H}_{X,2n}^A\colon A^{*,*}_\eta(X_+\wedge(\Proj^{2n-1}/(\Proj^{2n-1}-0_\Proj)))\xrightarrow{\simeq}A^{*,*}_\eta(X_+\wedge(\A^{2n}-0,1_\A)).
\]
\end{lemma}
\begin{proof}
The proof does not depend on the base, so we omit $X$ from notation. One may smash everything with $X_+$ and use the same the reasoning.

Put $Y=\left(\left(\A^{2n-2}\times (1,1)\right)\cup \left( (\A^{2n-2}-0)\times \A^2 \right)\right)\subset (\A^{2n}-0)$ and note that $Y$ can be contracted to a point in two steps: 
\[
\left(\A^{2n-2}\times (1,1)\right)\cup \left( (\A^{2n-2}-0)\times \A^2 \right) \sim \left(\A^{2n-2}\times (1,1)\right)\sim 1_\A.
\]
Consider the following diagram.
\[
\xymatrix @R=2pc @C=1pc{
 T^{\wedge 2n-2} \wedge (\Proj^1,[1:1]) \ar[d]_\simeq & T^{\wedge 2n-2}\wedge (\A^2-0,(1,1))  \ar[d]_i \ar[l]_{\id\wedge H_2} \ar[r]^(0.55)\simeq &  (\A^{2n}-0)/Y   \ar@/^4pc/[ddd]_{H_{2n}} \\
T^{\wedge 2n-2}\wedge (\Proj^1/\A^1) & T^{\wedge 2n-2}\wedge ((\A^2-0)/(\A^1\times \Gm))  \ar[l]_(0.6)*!/^-6pt/{\labelstyle \id\wedge H_2} \ar[ddr]^{H_{2n}}  & (\A^{2n}-0,1_\A) \ar[u]^\simeq \ar[dd]_{H_{2n}}  \\
T^{\wedge 2n-2}\wedge T \ar[u]^\simeq \ar@{=}[d] & T^{\wedge 2n-2}\wedge T\wedge \Gm_+   \ar[u]^{\simeq} \ar[dl]_(0.6){\id\wedge\phi} \ar[l]_(0.6){\id\wedge\psi} \ar[d]_{\Phi}&     \\
T^{\wedge 2n-2}\wedge T \ar@{=}[r]  &  T^{\wedge 2n-1}  \ar[r]^(0.4){j}_(0.4)\simeq &  \Proj^{2n-1}/(\Proj^{2n-1}-0_\Proj)
}
\]
Here all the arrows marked with "$\simeq$" are induced by the tautological embeddings and these morphisms are isomorphisms by excision or $\A^1$-contractibilty. Morphism $i$ is also induced by the tautological inclusion. All the maps denoted by $H_{2n}$ are induced by $H_{2n}$ defined in the statement of the lemma. Morphisms $\psi,\phi\colon T\wedge(\Gm_+,+)\to T$ are given by $\psi(x,t)=x/t$ and $\phi(x,t)=x/t^{2n-1}$ respectively. Finally, $j$ is given by $j(x_1,x_2,\dots,x_{2n-1})=[x_1:x_2:\dots:x_{2n-1}:1]$ and $\Phi$ is given by $\Phi(x_1,x_2,\ldots,x_{2n-1},t)=(x_1/t,x_2/t,\ldots,x_{2n-1}/t)$.

A straightforward check shows that the diagram commutes except possibly for the lower-left square involving $\psi,\phi$ and $\Phi$. Morphisms $\id\wedge \phi$ and $\Phi$ coincide in $\Hp(k)$ by Lemma~\ref{l:elematrix} since they differ by the automorphism of $T^{\wedge 2n-1}\wedge \Gm_+=\Th(\triv^{2n-1}_{\Gm})$ given by diagonal matrix $\operatorname{diag}(t,t,\ldots,t,1/t^{2n-2})$. Lemma~\ref{lemma_square_symbol} yields that for 
\[
f_{t^{2(n-1)}}\colon T\wedge\Gm_+\to T\wedge \Gm_+,\quad f_{t^{2(n-1)}}(x,t)=(t^{2(n-1)}x,t),
\]
we have
\[
\Sigma_T^\infty \psi = \Sigma_T^\infty (\phi f_{t^{2(n-1)}}) = (\Sigma_T^\infty \phi) \circ (\Sigma_T^\infty f_{t^{2(n-1)}}) = \Sigma_T^\infty \phi.
\]
Hence the $\Sigma_T^\infty$-suspension of the diagram commutates. The claim of the lemma follows from the diagram chase combined with the observation that 
\[
(\Sigma_T^{2n-1}\eta)^A=(\id\wedge H_2)^A\colon
A^{*,*}_\eta (T^{\wedge 2n-2} \wedge (\Proj^1,[1:1])) \to A^{*,*}_\eta (T^{\wedge 2n-2}\wedge (\A^2-0,(1,1)))
\]
is an isomorphism since we have inverted $\eta$ in the coefficients.
\end{proof}	

\begin{theorem}\label{Th:proj_coh}
Let $A$ be a commutative ring $T$-spectrum and let $X$ be a smooth variety. Then
\begin{enumerate}
\item 
The projection $p\colon \Proj^{2n}_X\to X$ induces an isomorphism $A^{*,*}_\eta(\Proj^{2n}_X) \cong A^{*,*}_\eta(X)$;
\item
The projection $H_{2n}$ induces an isomorphism
\[
H_{X,2n}^A\colon A^{*,*}_\eta(\Proj^{2n-1}_X)\xrightarrow{\simeq} A^{*,*}_\eta(\A^{2n}_X-(0\times X)),
\]
and, together with a choice of a section $s\colon X\to (\A^{2n}_X-(0\times X))$, induces an isomorphism
$A^{*,*}_\eta(\Proj^{2n-1}_X)\cong A^{*,*}_\eta(X)\oplus A^{*-4n+1,*-2n}_\eta(X).$
\end{enumerate}
\end{theorem}
\begin{proof}
The proof does not depend on the base, so we omit it from the notation. Note that the first claim is equivalent to $A^{*,*}_\eta(\Proj^{2n},x)=0$ for a rational point $x\in \Proj^{2n}$ and the second one is equivalent to the claim that 
\[
H_{2n}^A\colon A^{*,*}_\eta(\Proj^{2n-1},H_{2n}(y))\xrightarrow{} A^{*,*}_\eta(\A^{2n}-0,y)
\]
is an isomorphism for a rational point $y\in (\A^{2n}-0)$.

Proceed by induction: the case of $\Proj^0$ is trivial. Denote
\[
1_{\Proj^k}=[1:1:\ldots:1]\in\mathbb{P}^{k},\quad 0_{\Proj^k}=[0:\ldots:0:1]\in\mathbb{P}^{k},\quad
1_{\A^k}=(1,1,\ldots, 1)\in \A^{k}-0.
\]

$\mathbf{2n\to 2n+1}$: 
Consider the long exact sequence associated to the closed embedding $(\Proj^{2n},1_{\Proj^{2n}})\to (\Proj^{2n+1},1_{\Proj^{2n+1}})$ given by $x\mapsto [1:x]$:
\[
\ldots\xrightarrow{} A^{*-1,*}_\eta(\Proj^{2n},1_{\Proj^{2n}})\xrightarrow{\partial} A^{*,*}_\eta(\Proj^{2n+1}/\Proj^{2n})\xrightarrow{r^A} A^{*,*}_\eta(\Proj^{2n+1},1_{\Proj^{2n+1}}) \xrightarrow{} A^{*,*}_\eta(\Proj^{2n},1_{\Proj^{2n}}) \xrightarrow{\partial}\ldots
\]
By the induction assumption we know that $A^{*,*}_\eta(\Proj^{2n},1_{\Proj^{2n}})=0$, thus $r^A$ is an isomorphism. Moreover, $\Proj^{2n+1}-0_{\Proj^{2n+1}}\cong \struct_{\Proj^{2n}}(1)$ is a vector bundle over $\Proj^{2n}$ yielding 
\[
A^{*,*}_\eta(\Proj^{2n+1}/(\Proj^{2n+1}-0_{\Proj^{2n+1}}))\cong A^{*,*}_\eta(\Proj^{2n+1}/\Proj^{2n}).
\]
Applying Lemma~\ref{HopfSuspended} we obtain:
\[
A^{*,*}_\eta(\Proj^{2n+1},1_{\Proj^{2n+1}})\cong A^{*,*}_\eta(\Proj^{2n+1}/\Proj^{2n})\cong A^{*,*}_\eta(\Proj^{2n+1}/(\Proj^{2n+1}-0_{\Proj^{2n+1}})) \cong A^{*,*}_\eta(\A^{2n+2}-0,1_{\A^{2n+2}}).
\]
One can easily check that the above composition is given precisely by $H_{2n+2}^A$. The canonical isomorphism $(\A^{2n+2}-0,1_{\A^{2n+2}})\cong S^{4n+3,2n+2}$ (see \cite[Example 2.20]{MV99}) yields $A^{*,*}_\eta(\A^{2n+2}-0,1_{\A^{2n+2}}) \cong A^{*-4n-3,*-2n-2}_\eta(\pt)$.

$\mathbf{2n-1\to 2n}$:  
Recall that the complement to the zero section of $\struct_{\Proj^{2n-1}}(1)$ is isomorphic to $(\A^{2n}-0)$ and the projection composed with this isomorphism gives a projection $(\A^{2n}-0)\to \Proj^{2n-1}$ which coincides with $H_{2n}$. The long exact sequence
\[
\xymatrix @R=1pc @C=1pc{
\ldots \ar[r] & A^{*,*}_\eta(\Th(\struct_{\Proj^{2n-1}}(1))) \ar[r] & A^{*,*}_\eta(\struct_{\Proj^{2n-1}}(1)) \ar[r]  & A^{*,*}_\eta(\struct_{\Proj^{2n-1}}(1)-\Proj^{2n-1}) \ar[r] & \ldots \\
& & A^{*,*}_\eta(\Proj^{2n-1}) \ar@{=}[u] \ar[r]^{H_{2n}^A}_\simeq & A^{*,*}_\eta(\A^{2n}-0) \ar@{=}[u]& 
}
\]
yields $A^{*,*}_\eta(\Th(\struct_{\Proj^{2n-1}}(1)))=0$. Embed $\A^{2n}\to\Proj^{2n},\, x\mapsto [x:1]$. Excision provides an isomorphism
\[
A^{*,*}_\eta(\Proj^{2n}/\A^{2n}) \cong A^{*,*}_\eta(\Th(\struct_{\Proj^{2n-1}}(1)))=0.
\] 
$\A^1$-contractibility of $\A^{2n}$ gives us the claim: 
\[
A^{*,*}_\eta(\Proj^{2n},1_{\Proj^{2n}}) \cong A^{*,*}_\eta(\Proj^{2n}/\A^{2n})=0.\qedhere
\]
\end{proof}

\begin{corollary}
Let $A$ be a commutative ring $T$-spectrum and let $E$ be a vector bundle of odd rank over a smooth variety $X$. Then canonical projection $p\colon \Proj_X(E)\to X$ induces an isomorphism 
\[
p^A\colon A^{*,*}_\eta(X)\xrightarrow{\simeq} A^{*,*}_\eta(\Proj_X(E)).
\]
\end{corollary}
\begin{proof}
Follows via Mayer-Vietoris long exact sequence.
\end{proof}

\begin{corollary}
\label{c:twistedpbundle}
Let $A$ be a commutative ring $T$-spectrum and let $V$ be a vector bundle over a smooth variety $X$. Then projection $p\colon \Proj^{2n}_X\to X$ induces an isomorphism $p^A\colon A^{*,*}_\eta(X;V) \xrightarrow{\simeq} A^{*,*}_\eta(\Proj^{2n}_X; p^*V)$.
\end{corollary}

\begin{proof}
Consider the following diagram consisting of long exact sequences associated to the zero sections of $V$ and $p^*V$.
\[
\xymatrix @R=1pc @C=1pc{
\dots \ar[r] & A_\eta^{*-2n,*-n}(X;V) \ar[r]\ar[d]^{p^A} & A_\eta^{*,*}(X) \ar[r]\ar[d]_\simeq^{p^A} & A_\eta^{*,*}(V- X) \ar[r] \ar[d]_\simeq^{p^A} & \dots \\
\dots \ar[r] & A_\eta^{*-2n,*-n}(\Proj^{2n}_X;p^*V) \ar[r] & A_\eta^{*,*}(\Proj^{2n}_X) \ar[r] & A_\eta^{*,*}(p^*V -\Proj^{2n}_X ) \ar[r] & \dots 
}
\]
Note that $p^*V-\Proj^{2n}_X=(V-X)\times \mathbb{P}^{2n}$, so the second and the third vertical maps are isomorphisms by Theorem~\ref{Th:proj_coh}. The claim follows via the five lemma.
\end{proof}

\begin{theorem}
\label{t:twistedpbundlevect}
Let $A$ be a commutative ring $T$-spectrum and let $X$ be a smooth variety. Denote $p_1\colon \mathbb{P}_X^k\to X$ and $p_2\colon \mathbb{P}^k_X\to \mathbb{P}^k$ the canonical projections. Consider a vector bundle $V$ over $X$, a vector bundle $E$ over $\Proj^k$ of degree $d= \deg \det E$ and a rational point $a\in \Proj^{k}$. Denote $i\colon X\to \mathbb{P}^k_X$ the closed embedding given by $i(x)=(x,a)$. Then we have the following isomorphisms depending on the parity of $d$ and $k$.
\begin{itemize}
\item[Ia)] If $d=2m-1,\, k=2n-1$ then $A^{*,*}_\eta(\Proj^{k}_X;p_1^*V\oplus p_2^*E)=0$;
\item[Ib)] If $d=2m,\, k=2n-1$ then there is a split short exact sequence 
\[
A^{*-2k,*-k}_\eta(X;i^*(p_1^*V\oplus p_2^*E)\oplus N_{i})\xrightarrow{i_A} A^{*,*}_\eta(\Proj^{k}_X;p_1^*V\oplus p_2^*E)\xrightarrow{i^A} A^{*,*}_\eta(X;i^*(p_1^*V\oplus p_2^*E));
\]
\item[IIa)] If $d=2m-1,\, k=2n$ then 
\[
A^{*-2k,*-k}_\eta(X;i^*(p_1^*V\oplus p_2^*E)\oplus N_{i})\xrightarrow{i_A} A^{*,*}_\eta(\Proj^{k}_X;p_1^*V\oplus p_2^*E)
\]
is an isomorphism;
\item[IIb)] If $d=2m,\, k=2n$ then 
\[
A^{*,*}_\eta(\Proj^{k}_X;p_1^*V\oplus p_2^*E)\xrightarrow{i^A}A^{*,*}_\eta(X;i^*(p_1^*V\oplus p_2^*E))
\]
is an isomorphism.
\end{itemize}
\end{theorem}
\begin{proof}

The proof does not depend on the base $X$ and vector bundle $V$, so we omit them from notation and suppose that $X=\pt$, $V=0$ and $p_2=\id$. In case of nontrivial $X$ and $V$ everything is virtually the same except that one should use Corollary~\ref{c:twistedpbundle} in place of Theorem~\ref{Th:proj_coh}.

The proof goes as follows. First we obtain { Ia)} for $\Proj^1$ and using an induction step from $k=2n-1$ to $k=2n+1$ obtain {Ia)} in general. Then we easily deduce {IIa)} and {IIb)} from {Ia)}. The most tricky part {Ib)} is dealt in the end, we derive it from {Ia)} and {IIb)} for $k=2n-2$.

Throughout the proof we continuously use long exact sequences of the following kind. Consider a linear embedding $r\colon \Proj^{l'} \to \Proj^l$. The open complement $q\colon \Proj^{l}-\Proj^{l'}\subset \Proj^l$ is a vector bundle over $\Proj^{l-l'-1}$ and excision shows that $\Proj^{l}/(\Proj^{l}-\Proj^{l'})$ is isomorphic to $\Th(\struct_{\Proj^{l'}}(1)^{\oplus (l-l')})$. Thus for a vector bundle $\mathcal{E}$ over $\Proj^l$ we have a long exact sequence
\[
\dots \to A^{*-2l+2l',*-l+l'}_\eta (\Proj^{l'};r^*\mathcal{E}\oplus \struct_{\Proj^{l'}}(1)^{\oplus (l-l')})\xrightarrow{r_A} A^{*,*}_\eta (\Proj^{l};\mathcal{E})
\xrightarrow{q^A} A^{*,*}_\eta (\Proj^{l-l'-1};q^*\mathcal{E})\to\dots
\]

{\bf Ia) for $\Proj^1$.} Embed $\Proj^1$ to $\Proj^{2m}$. We have a long exact sequence
\[
\dots \to A^{*-4m+2,*-2m+1}_\eta (\Proj^1;\struct_{\Proj^1}(1)^{\oplus 2m-1})\to A^{*,*}_\eta (\Proj^{2m})\xrightarrow{q^A} A^{*,*}_\eta (\Proj^{2m-2})\to \dots
\]
Theorem~\ref{Th:proj_coh} yields that $q^A$ is an isomorphism and $A^{*,*}_\eta (\Proj^1;\struct_{\Proj^1}(1)^{\oplus 2m-1})=0$. Corollary~\ref{c:K0_coh} combined with Lemma~\ref{l:dual_coh} and the fact that 
\[
K_0(\Proj^1)=\mathbb{Z}\cdot [\triv_{\Proj^1}]\oplus \mathbb{Z}\cdot [\struct_{\Proj^1}(1)]
\]
gives claim Ia) for $\Proj^1$.

{\bf Ia) for $\mathbf{\Proj^{2n-1}}$ $\Rightarrow$ Ia) for $\mathbf{\Proj^{2n+1}}$.} Consider a linear embedding $r\colon \Proj^1 \to \Proj^{2n+1}$ and the corresponding long exact sequence
\[
\dots \to A^{*-4n,*-2n}_\eta (\Proj^1;r^*E\oplus \struct_{\Proj^1}(1)^{\oplus 2n})\to A^{*,*}_\eta (\Proj^{2n+1};E) 
\xrightarrow{q^A} A^{*,*}_\eta (\Proj^{2n-1};q^*E)\to \dots
\]
By the induction assumption we have $A^{*,*}_\eta (\Proj^{2n-1};q^*E)=0$ and $A^{*,*}_\eta (\Proj^1;r^*E\oplus \struct_{\Proj^1}(1)^{\oplus 2n})=0$ , so the claim follows.

{\bf Ia)$\Rightarrow$ IIa).} Consider the long exact sequence associated to the embedding $i\colon \Proj^{0}\to \Proj^{2n}$.
\[
\dots \to A^{*-4n,*-2n}_\eta (\Proj^{0};i^*E\oplus \struct_{\Proj^{0}}(1)^{\oplus 2n})\xrightarrow{i_A} A^{*,*}_\eta (\Proj^{2n};E)
\xrightarrow{q^A} A^{*,*}_\eta (\Proj^{2n-1};q^*E)\to \dots
\]
Assumption Ia) yields that $A^{*,*}_\eta (\Proj^{2n-1};q^*E)=0$, thus $i_A$ is an isomorphism.

{\bf Ia)$\Rightarrow$ IIb).} Consider a linear embedding $r\colon \Proj^{2n-1}\to \Proj^{2n}$ and the corresponding long exact sequence
\[
\dots \to A^{*-2,*-1}_\eta (\Proj^{2n-1};r^*E\oplus \struct_{\Proj^{2n-1}}(1))\to A^{*,*}_\eta (\Proj^{2n};E)
 \xrightarrow{q^A} A^{*,*}_\eta (\Proj^0;q^*E)\to \dots
\]
Assumption Ia) yields that $A^{*-2,*-1}_\eta (\Proj^{2n-1};r^*E\oplus \struct_{\Proj^{2n-1}}(1))=0$, thus $q^A$ is an isomorphism and a choice of a rational point on $\Proj^{2n}$ provides an inverse.

{\bf Ia) and {\bf IIb)} for $\Proj^{2n-2}$ $\Rightarrow$ Ib).} Consider the long exact sequence associated to the embedding $i\colon \Proj^0 \to \Proj^{2n-1}$.
\[
\dots \to A^{*-4n+2,*-2n+1}_\eta (\Proj^{0};i^*E\oplus \struct_{\Proj^{0}}(1)^{\oplus 2n-1})\xrightarrow{i_A} A^{*,*}_\eta (\Proj^{2n-1};E) 
\xrightarrow{q^A} A^{*,*}_\eta (\Proj^{2n-2};q^*E)\to \dots
\]
Assumption {IIb)} for $\Proj^{2n-2}$ yields $A^{*,*}_\eta (\Proj^{2n-2};q^*E)\cong A^{*,*}_\eta(\pt;i^*E)$. Thus we have a long exact sequence
\[
\dots \to A^{*-4n+2,*-2n+1}_\eta(\pt;i^*E\oplus N_i)\xrightarrow{i_A} A^{*,*}_\eta(\Proj^{2n-1};E)
\xrightarrow{i^A}  A^{*,*}_\eta(\pt;i^*E)\to \dots
\]
Hence it is sufficient to show that $i^A$ is a split surjection.

Consider the twisted version of the localization sequence for the complement to the zero section of $\struct_{\Proj^{2n-1}}(1)$ (recall that we denote the projection for the complement $H_{2n}\colon \mathbb{A}^{2n}-0 \to \Proj^{2n-1}$):
\[
\dots \to A^{*-2,*-1}_\eta (\Proj^{2n-1};E\oplus \struct_{\Proj^{2n-1}}(1))\to A^{*,*}_\eta (\Proj^{2n-1};E) 
\xrightarrow{H_{2n}^A}  A^{*,*}_\eta (\mathbb{A}^{2n}-0;H_{2n}^*E)\to \dots
\]
Assumption Ia) gives $A^{*-2,*-1}_\eta (\Proj^{2n-1};E\oplus \struct_{\Proj^{2n-1}}(1))=0$, thus $H_{2n}^A$ is an isomorphism. Hence it is sufficient to show that 
\[
j^A\colon A^{*,*}_\eta (\mathbb{A}^{2n}-0;H_{2n}^*E) \to A^{*,*}_\eta (\pt;j^*H_{2n}^*E)
\]
is a split surjection for an embedding $j\colon \pt \to \mathbb{A}^{2n}-0$. This basically follows from Corollary~\ref{c:K0_coh} and the fact that $\widetilde{K}_0(\mathbb{A}^{2n}-0)=0$. More precisely, every vector bundle over $Y=\mathbb{A}^{2n}-0$ is stably trivial, thus we may find some $s,t$ and an isomorphism $\theta\colon H_{2n}^*E\oplus \triv_{Y}^s \cong\triv_{Y}^t$. Corollary~\ref{c:suspuntwist} gives us a commutative diagram
\[
\xymatrix @R=1pc @C=1pc{
A^{*,*}_\eta (Y;H_{2n}^*E) \ar[r]^(0.42)\simeq \ar[d]^{j^A} & A^{*,*}_\eta (Y;H_{2n}^*E\oplus \triv_{Y}^s) \ar[r]_(0.60)\simeq^(0.60){\theta^A} \ar[d]^{j^A} & A^{*,*}_\eta (Y; \triv_{Y}^t) \ar[d]^{j^A}  & A^{*,*}_\eta (Y) \ar[l]_(0.4)\simeq \ar[d]^{j^A} \\
A^{*,*}_\eta (\pt;j^*H_{2n}^*E) \ar[r]^(0.42)\simeq & A^{*,*}_\eta (\pt;j^*H_{2n}^*E\oplus \triv_{\pt}^s) \ar[r]_(0.60)\simeq^(0.65){j^*\theta^A} & A^{*,*}_\eta (\pt;\triv_{\pt}^t) & A^{*,*}_\eta (\pt) \ar[l]_(0.4)\simeq
}
\]
We denote the vertical morphisms by the same letter since all of them are induced by the embedding $j$. The rightmost vertical morphism is clearly surjective with a splitting given by $p^A$, where $p\colon Y\to \pt$ is the projection. Hence the leftmost vertical morphism is a split surjection as well and we get the claim.
\end{proof}

We sum up our computations in the following form. Note that the isomorphisms in the following Corollary depend on the choice of trivializations of certain bundles.

\begin{corollary} \label{cor:twistedpbundlevect}
In notation of Theorem~\ref{t:twistedpbundlevect} we have the following isomorphisms depending on the parity of $d$ and $k$.
\begin{itemize}
\item[Ia)] $d=2m-1,\, k=2n-1\colon\, A^{*,*}_\eta(\Proj^{k}_X;p_1^*V\oplus p_2^*E)=0$.
\item[Ib)] $d=2m,\, k=2n-1\colon
A^{*,*}_\eta(\Proj^{k}_X;p_1^*V\oplus p_2^*E)\cong A^{*,*}_\eta(X;V)\oplus A^{*-2k,*-k}_\eta(X;V)$.
\item[IIa)] $d=2m-1,\, k=2n\colon\, A^{*,*}_\eta(\Proj^{k}_X;p_1^*V\oplus p_2^*E)\cong A^{*-2k,*-k}_\eta(X;V)$.
\item[IIb)] $d=2m,\, k=2n\colon\, A^{*,*}_\eta(\Proj^{k}_X;p_1^*V\oplus p_2^*E)\cong A^{*,*}_\eta(X;V)$.
\end{itemize}
\end{corollary}
\begin{proof}
The claim follows from Theorem~\ref{t:twistedpbundlevect} combined with Corollary~\ref{c:K0_coh} and isomorphisms $N_{i}\cong \triv_X^{k},\, i^*(p_1^*V\oplus p_2^*E)\cong V\oplus \triv_X^{\rank E}$.
\end{proof}

\section{Cohomology twisted by a formal difference of vector bundles}

In this section we introduce twists by a formal difference of vector bundles and establish its basic properties. Roughly speaking, in order to define cohomology groups twisted by formal differences we add a trivialized vector bundle of large rank and use Definition~\ref{d:tw}. Keeping track of all the isomorphisms allows us to obtain functoriality. 

\begin{definition} \label{def:fdtw}
Let $E'$ and $E$ be vector bundles over a smooth variety $X$. Applying Jouanolou's device (\cite{J73}, \cite[\S 4]{We89}) we may assume that $X$ is affine. For a vector bundle $\overline{E}$ and an isomorphism $\theta \colon E\oplus \overline{E}\xrightarrow{\simeq} \triv_X^{2n}$ put
\[
A^{*,*}_{(\overline{E},\theta)}(X;E'\ominus E)=A^{*,*}(X;E'\oplus \overline{E}).
\]
Let $\rho\colon E\oplus\widetilde{E}\xrightarrow{\simeq} \triv_X^{2m}$ be another isomorphism. Define \textit{canonical isomorphism}
\[
\Theta_{(\overline{E},\theta)}^{(\widetilde{E},\rho)}\colon A^{*,*}(X;E'\oplus \overline{E})\xrightarrow{\simeq} A^{*,*}(X;E'\oplus \widetilde{E})
\]
to be given by the following sequence of isomorphisms:
\[
\xymatrix @C=8pc{
A^{*,*}(X;E'\oplus \overline{E}) \ar[d]_{-\cup\Sigma_T^{2m} 1} & A^{*,*}(X;E'\oplus \widetilde{E}) \ar[d]^{-\cup\Sigma_T^{2n} 1} \\
A^{*,*}(X;E'\oplus \overline{E}\oplus \triv_X^{2m}) \ar[d]_{(\id\oplus \rho)^A} & A^{*,*}(X;E'\oplus \widetilde{E} \oplus \triv_X^{2n}) \ar[d]^{(\id\oplus \theta)^A} \\
A^{*,*}(X;E'\oplus \overline{E}\oplus E \oplus \widetilde{E}) \ar[r]^{(\id\oplus \tau(\overline{E},\widetilde{E}))^{A}} & A^{*,*}(X;E'\oplus \widetilde{E}\oplus E \oplus \overline{E})
}
\]
Here $\tau(\overline{E},\widetilde{E})\colon \widetilde{E}\oplus E \oplus \overline{E}\to \overline{E}\oplus E \oplus \widetilde{E}$ is given by the matrix:
\[
\tau=\begin{pmatrix} 0 & 0 & \id_{\overline{E}}\\ 0 & \id_{E} & 0\\ -\id_{\widetilde{E}} & 0 & 0 \end{pmatrix}.
\]
A straightforward computation shows that these isomorphisms are functorial in $X$ and satisfy 
\begin{enumerate}
\item
$\Theta_{(\overline{E},\theta)}^{(\overline{E},\theta)}=\id$;
\item
$\Theta_{(\overline{E}_2,\theta_2)}^{(\overline{E}_3,\theta_3)}\circ \Theta_{(\overline{E}_1,\theta_1)}^{(\overline{E}_2,\theta_2)}=\Theta_{(\overline{E}_1,\theta_1)}^{(\overline{E}_3,\theta_3)}$ for isomorphisms $\theta_i \colon E\oplus \overline{E}_i\xrightarrow{\simeq} \triv_X^{2n_i}$.
\end{enumerate}
These isomorphisms allows us to identify canonically the twisted cohomology groups for different choices of complement bundles. We sometimes omit the explicit choice of $(\overline{E},\theta)$ from the notation and write
\[
A^{*,*}(X;E'\ominus E)=A^{*,*}_{(\overline{E},\theta)}(X;E'\ominus E)=A^{*,*}(X;E'\oplus \overline{E}).
\]
These groups are referred to as \textit{cohomology groups of $X$, twisted by the formal difference of vector bundles $E_1\ominus E_2$}. Pull-back homomorphisms, cup-product and push-forward maps along closed embeddings for the cohomology groups twisted by vector bundles induce respective structures for the cohomology groups twisted by a formal difference of vector bundles provided that the choices of complement bundles agree. Moreover, all these maps respect the canonical isomorphisms. For example, for a closed embedding of smooth varieties $i\colon Y\to X$ and a formal difference of vector bundles $\mathcal{E}=E'\ominus E$ over $X$ the push-forward map described in the end of Definition~\ref{d:tw} gives rise to the push-forward map
\[
\xymatrix @C=1pc{
A^{*,*}_{(i^*\overline{E},i^*\theta)}(Y;i^*\mathcal{E}\oplus N_i) \ar@{=}[r] \ar[d]^{i^{(\overline{E},\theta)}_A} & A^{*,*}(Y;i^*E'\oplus N_i\oplus i^*\overline{E})  \ar[d]^{i_A} \\
A^{*+2n,*+n}_{(\overline{E},\theta)}(X;\mathcal{E})  \ar@{=}[r] & A^{*+2n,*+n}(X;E'\oplus \overline{E})
}
\]
satisfying $\Theta_{(\overline{E},\theta)}^{(\widetilde{E},\rho)}\circ i^{(\overline{E},\theta)}_A=i^{(\widetilde{E},\rho)}_A\circ \Theta_{(i^*\overline{E},i^*\theta)}^{(i^*\widetilde{E},i^*\rho)}$. Having the last equality in mind we usually omit $(\overline{E},\theta)$ from the notation and denote the push-forward map just as 
\[
A^{*,*}(Y;i^*\mathcal{E}\oplus N_i)\xrightarrow{i_A} A^{*+2n,*+n}(X;\mathcal{E}).
\]
\end{definition}

\begin{remark}
For an oriented cohomology theory one has natural isomorphisms 
\[
A^{*,*}(X;E_1\ominus E_2)\cong A^{*,*}(X).
\]
For an $\SL$-oriented cohomology theory Corollary~\ref{c:sluntwist} yields 
\[
A^{*,*}(X;E_1\ominus E_2)\cong A^{*,*}(X;\det E_1 \otimes (\det E_2)^{-1}).
\]
\end{remark}

\begin{lemma}
\label{l:tdc}
Let $\mathcal{V}=V'\ominus V$ be a formal difference of vector bundles over a smooth variety $X$, let $E$ be a vector bundle over $X$ and let $\theta \colon V\oplus \overline{V}\xrightarrow{\simeq} \triv_X^{2n}$ and $\rho \colon E\oplus \overline{E}\xrightarrow{\simeq} \triv_X^{2m}$ be isomorphisms of vector bundles. Then there is an isomorphism
\[
\Xi_{(\overline{V},\theta)}^{(\overline{E},\rho)}\colon A^{*,*}_{(\overline{V},\theta)}(X;\mathcal{V}) \xrightarrow{\simeq} A^{*,*}_{(\overline{V}\oplus \overline{E},\theta\oplus \rho)}(X;\mathcal{V} \oplus E\ominus E)
\]
functorial in $X$ and agreeing with canonical isomorphisms in a sense that 
\[
\Xi_{(\widetilde{V},\theta')}^{(\widetilde{E},\rho')}\circ \Theta_{(\overline{V},\theta)}^{(\widetilde{V},\theta')} =  \Theta_{(\overline{V}\oplus\overline{E},\theta\oplus \rho)}^{(\widetilde{V}\oplus\widetilde{E},\theta'\oplus\rho')} \circ \Xi_{(\overline{V},\theta)}^{(\overline{E},\rho)}
\]
for another choice of complement bundles.
\end{lemma}
\begin{proof}
Denote
\[
\thc(\overline{E},\rho)= \rho^A(\Sigma_T^{2m} 1)\in A^{*,*}(X; E\oplus\overline{E}).
\]
The desired isomorphism is given by the cup-product with $\thc(\overline{E},\rho)$:
\[
A^{*,*}(X;V'\oplus\overline{V}) \xrightarrow{- \cup \thc(\overline{E},\rho)} A^{*,*}(X;V'\oplus\overline{V}\oplus E \oplus\overline{E}).
\]
The functoriality property and consistency with canonical isomorphisms are straightforward.
\end{proof}

\begin{definition}
Let $i\colon Y\to X$ be a closed embedding of smooth varieties with a rank $n$ normal bundle $N_i$ and let $\mathcal{E}=E'\ominus E$ be a formal difference of vector bundles over $X$. For isomorphisms of vector bundles $\theta \colon E\oplus \overline{E}\xrightarrow{\simeq} \triv_X^{2m_1}$, $\tau_X \colon T_X\oplus \overline{T}_X\xrightarrow{\simeq} \triv_X^{2m_2}$ and $\tau_Y \colon T_Y\oplus \overline{T}_Y\xrightarrow{\simeq} \triv_Y^{2m_3}$ define push-forward map
\[
i^{\theta,\tau_X,\tau_Y}_A\colon A^{*,*}_{(i^*\overline{E}\oplus\overline{T}_Y,i^*\theta\oplus\tau_Y)}(Y;i^*\mathcal{E}\ominus T_Y)\to A^{*+2n,*+n}_{(\overline{E}\oplus\overline{T}_X,\theta\oplus\tau_X)}(X;\mathcal{E}\ominus T_X)
\]
as the composition
\begin{multline*}
A^{*,*}_{(i^*\overline{E}\oplus\overline{T}_Y,i^*\theta\oplus\tau_Y)}(Y;i^*\mathcal{E}\ominus T_Y) \xrightarrow{\simeq} \\
\xrightarrow{\simeq} A^{*,*}_{(i^*\overline{E}\oplus i^*\overline{T}_X\oplus\overline{T}_Y,i^*\theta\oplus i^*\tau_X\oplus\tau_Y)}(Y;i^*\mathcal{E} \ominus i^*T_X\oplus i^*T_X \ominus T_Y) \xrightarrow{\simeq} \\
\xrightarrow{\simeq} A^{*,*}_{(i^*\overline{E}\oplus i^*\overline{T}_X\oplus\overline{T}_Y,i^*\theta\oplus i^*\tau_X\oplus\tau_Y)}(Y;i^*\mathcal{E}\ominus i^*T_X \oplus N_i\oplus T_Y \ominus T_Y) \xrightarrow{\simeq} \\
\xrightarrow{\simeq} A^{*,*}_{(i^*\overline{E}\oplus i^*\overline{T}_X,i^*\theta\oplus i^*\tau_X)}(Y;i^*\mathcal{E}\ominus i^*T_X \oplus N_i)  \xrightarrow{i^{(\overline{E}\oplus\overline{T}_X,\theta\oplus \tau_X)}_A} \\
\xrightarrow{i^{(\overline{E}\oplus\overline{T}_X,\theta\oplus \tau_X)}_A}  A^{*+2n,*+n}_{(\overline{E}\oplus\overline{T}_X,\theta\oplus\tau_X)}(X;\mathcal{E}\ominus T_X)
\end{multline*}
Here the first and the third isomorphisms are given by Lemma~\ref{l:tdc}, the second one is given by Lemma~\ref{l:twq} applied to the short exact sequence $T_Y\to i^*T_X \to N_i$ and the last map is the push-forward map introduced in the end of Definition~\ref{def:fdtw}. One can easily check that morphism $i^{\theta,\tau_X,\tau_Y}_A$ agrees with the canonical isomorphisms:
\[
i^{\theta',\tau'_X,\tau'_Y}_A \circ \Theta_{(i^*\overline{E}\oplus \overline{T}_Y,i^*\theta\oplus \tau_Y)}^{(i^*\widetilde{E}\oplus \widetilde{T}_Y,i^*\theta'\oplus \tau'_Y)} = 
\Theta_{(\overline{E}\oplus \overline{T}_X,\theta\oplus \tau_X)}^{(\widetilde{E}\oplus \widetilde{T}_X,\theta'\oplus \tau'_X)} \circ i^{\theta,\tau_X,\tau_Y}_A 
\]
for another choice of complement bundles, which allows us to omit the choices from the notation and denote the push-forward map
\[
i_A\colon A^{*,*}(Y;i^*\mathcal{E}\ominus T_Y)\to A^{*+2n,*+n}(X;\mathcal{E}\ominus T_X).
\] 
\end{definition}

\begin{theorem}
\label{t:twistedpbundle}
All statements of Theorem~\ref{t:twistedpbundlevect} and Corollary~\ref{cor:twistedpbundlevect} hold for $V$ and $E$ being formal differences of vector bundles. In particular, for a formal difference $\mathcal{V}$ of vector bundles over $X$ the push-forward map
\[
A^{*-4n,*-2n}_\eta(X;\mathcal{V}\ominus T_{X})\xrightarrow{i_A} A^{*,*}_\eta(\Proj^{2n}_X;p_1^*\mathcal{V}\ominus T_{\Proj^{2n}_X})
\]
is an isomorphism.
\end{theorem}
\begin{proof}
Follows from Theorem~\ref{t:twistedpbundlevect}.
\end{proof}

\section{Push-forwards along projective morphisms}

In this section we define push-forward maps along projective morphisms following the strategy realized in \cite{Ne09a} for derived Witt groups. We formulate some useful properties of these maps. The proofs are quite similar to the ones given in loc. cit. and we leave it to the reader to modify the reasoning to fit in the considered setting. Moreover, the reader should keep in mind that in order to define push-forward maps one needs to make some explicit choices of complement vector bundles as described in the previous section, although, up to canonical isomorphisms from Definition~\ref{def:fdtw}, the result does not depend on the choices made.

\begin{lemma}[cf. {\cite[Lemma 3.1 (ii)]{Ne09a}}]
\label{l:pushpoint}
Let $\mathcal{V}$ be a formal difference of vector bundles over a smooth variety $X$. Consider two linear embeddings $j_1,j_2\colon\mathbb{P}^m\to \mathbb{P}^{k}$ and denote by $p_k\colon \mathbb{P}^k_X\to X$ and $p_m\colon \mathbb{P}^m_X\to X$ the canonical projections. Then 
\[
(j_1\times \id_X)_A=(j_2\times \id_X)_A\colon A^{*-2k,*-k}(\mathbb{P}^m_X;p_m^*\mathcal{V}\ominus T_{\mathbb{P}^{m}_X}) \to A^{*-2m,*-m}(\Proj^{k}_X;p_k^*\mathcal{V}\ominus  T_{\mathbb{P}^{k}_X}).
\]
\end{lemma}

\begin{definition} \label{d:prpush}
Let $\mathcal{V}$ be a formal difference of vector bundles over a smooth variety $X$ and let $p\colon \mathbb{P}^{k}_X\to X$ be the canonical projection.

\noindent $\mathbf{k=2n}$: choose a rational point $a\in \mathbb{P}^{k}$ and denote $i\colon X\to \mathbb{P}^{k}_X$ the embedding $i(x)=(x,a)$. Theorem~\ref{t:twistedpbundle} allows us to define
\[
p_A=i_A^{-1}\colon A^{*,*}_\eta(\Proj^{k}_X;p^*\mathcal{V}\ominus T_{\mathbb{P}^{k}_X})\xrightarrow{\simeq} A^{*-2k,*-k}_\eta(X;\mathcal{V}\ominus T_X).
\] 

\noindent $\mathbf{k=2n-1}$:
consider a linear embedding $j\colon \mathbb{P}^{k}\to \mathbb{P}^{k+1}$ and denote $\tilde{p}\colon \mathbb{P}_X^{k+1}\to X$ the canonical projection. Define
\[
p_A=\tilde{p}_A\circ j_A\colon A^{*,*}_\eta(\Proj^{k}_X;p^*\mathcal{V}\ominus T_{\mathbb{P}^{k}_X})\xrightarrow{} A^{*-2k,*-k}_\eta(X;\mathcal{V}\ominus T_X).
\] 

\noindent We refer to $p_A$ in both cases as \textit{push-forward map along the projection $p$}. By Lemma~\ref{l:pushpoint} the defined push-forwards $p_A$ do not depend on the choice of the rational point $a$ and the linear embedding $j$.
\end{definition}

\begin{definition} \label{def:projpush}
Let $\mathcal{V}$ be a formal difference of vector bundles over a smooth variety $X$ and let $f\colon Y\to X$ be a codimension $m$ projective morphism of smooth varieties. Represent $f$ as a composition
\[
f=p\circ i\colon Y\xrightarrow{i} \mathbb{P}^n_X \xrightarrow{p} X,
\]
where $i$ is a codimension $n+m$ closed embedding and $p$ is the canonical projection. We define \textit{push-forward along $f$} as the composition 
\[
f_A=p_A\circ i_A\colon A_\eta^{*,*}(Y;f^*\mathcal{V}\ominus T_Y) \to A_\eta^{*+2m,*+m}(X;\mathcal{V}\ominus T_X).
\]
Following {\cite[Proposition 4.2]{Ne09a}} one can show that $f_A$ does not depend on the chosen factorization $f=p\circ i$.
\end{definition}

\begin{lemma}[cf. {\cite[4.5]{Ne09a}}]
Let $\mathcal{V}$ be a formal difference of vector bundles over a smooth variety $X$ and let $Z\xrightarrow{g}Y\xrightarrow{f} X$ be projective morphisms of smooth varieties of codimensions $m$ and $n$. Then
\[
(fg)_A=f_A\circ g_A\colon A^{*,*}_\eta (Z;(gf)^*\mathcal{V}\ominus T_Z)\to A^{*+2(m+n),*+(m+n)}_\eta (X;\mathcal{V}\ominus T_X).
\]
\end{lemma}

\begin{lemma}[cf. {\cite[4.7]{Ne09a}}]
Consider the following Cartesian square consisting of smooth varieties with all the morphisms being smooth.
\[
\xymatrix{
Y' \ar[d]_{g'} \ar[r]^{f'} & X' \ar[d]^{g} \\
Y \ar[r]^{f} & X
}
\]
Suppose that the square is transversal, $f$ is projective and put $\codim f=\codim f'=n$. Let $\mathcal{V}$ be a formal difference of vector bundles over $X$. Then 
\[
g^A\circ f_A= f'_A\circ \phi \circ (g')^A\colon A_\eta^{*,*}(Y;f^*\mathcal{V}\ominus T_{Y})\to A_\eta^{*+2m,*+m}(X';g^*\mathcal{V}\ominus g^*T_X).
\]
Here 
\[
\phi\colon A_\eta^{*,*}(Y';(fg')^*\mathcal{V}\ominus (g')^*T_{Y})\xrightarrow{\simeq} A_\eta^{*,*}(Y';(fg')^*\mathcal{V}\oplus (g')^*T_{X'}\ominus ((fg')^*T_X\oplus T_{Y'}))
\]
is given by Lemma~\ref{l:twq} applied to the exact sequence
\[
0\to T_{Y'}\xrightarrow{(dg',-df')} (g')^*T_Y\oplus (f')^*T_{X'}\xrightarrow{(df,dg)} (fg')^*T_X\to 0
\]
and cancellation isomorphisms from Lemma~\ref{l:tdc}.
\end{lemma}

\begin{lemma}[cf. {\cite[4.11]{Ne09a}}]
Let $\mathcal{V}_1$ and $\mathcal{V}_2$ be formal differences of vector bundles over a smooth variety $X$ and let $f\colon Y\to X$ be a projective morphism of smooth varieties. Then
\[
f_A(f^A(\alpha)\cup \beta)=\alpha \cup f_A(\beta)
\]
for $\alpha\in A^{*,*}(X;\mathcal{V}_1),\,\beta\in A^{*,*}(Y;f^*\mathcal{V}_2\ominus T_{Y})$.
\end{lemma}

\textbf{Acknowledgements.}
The author wishes to thank I.~Panin for numerous conversations on the subject and the anonymous referee whose valuable comments hopefully made the paper much more readable. The research is supported by Russian Science Foundation grant 14-11-00456.


\begin{thebibliography}{XX}

\bibitem{An12}
Ananyevskiy~A.:
The special linear version of the projective bundle theorem.
Comp. Math. 151(3), 461--501 (2015)

\bibitem{ALP15}
Ananyevskiy~A., Levine~M., Panin~I.:
Witt sheaves and the $\eta$-inverted sphere spectrum.
arXiv:1504.04860, http://arxiv.org/abs/1504.04860

\bibitem{AM15}
Andrews~M., Miller~H.:
Inverting the Hopf map.
http://www-math.mit.edu/~hrm/papers/andrews-miller-jun14.pdf

\bibitem{Bal99}
Balmer~P.:
Derived Witt groups of a scheme.
J. Pure Appl. Algebra 141, 101---129 (1999)

\bibitem{BM00}
Barge J., Morel~F.:
Groupes de Chow des cycles orient\'es et classes d'Euler des fibr\'es vectoriels.
C.~R.~Math. Acad. Sci. 330, 287---290 (2000)

\bibitem{CH11}
Calm\`es~B., Hornbostel~J.:
Push-forwards for Witt groups of schemes.
Commentarii Mathematici Helvetici 86(2), 437---468 (2011)

\bibitem{G01}
Gille~S.:
A note on the Witt group of $\mathbb{P}^n$.
Math.~Z. 237, 601---619 (2001)

\bibitem{Jar00}
Jardine~J.~F.:
Motivic symmetric spectra.
Doc. Math. 5, 445---552 (2000)

\bibitem{J73}
Jouanolou~J.~P.:
Un suite exacte de Mayer--Vietoris en $K$--theorie algebrique.
Lecture Notes in Math. No. 341, Springer-Verlag (1973)

\bibitem{H05}
Hu~P.:
On the Picard group of the stable $\mathbb{A}^1$-homotopy category.
Topology 44(3), 609---640 (2005)

\bibitem{LM07}
Levine~M., Morel~F.:
Algebraic cobordism.
Springer (2007)

\bibitem{Mor04}
Morel~F.:
An introduction to $\A^1$-homotopy theory.
Contemporary developments in algebraic K-theory, ICTP Lecture Notes, vol. XV, (Abdus Salam International Center for Theoretical Physics, Trieste, 2004), 357--441.


\bibitem{MV99}
Morel~F., Voevodsky~V.:
$\A^1$-homotopy theory of schemes.
Publ. Math. IHES 90, 45---143 (1999)

\bibitem{Ne06}
Nenashev~A.:
Gysin maps in oriented theories.
J. of Algebra 302, 200---213 (2006)

\bibitem{Ne07}
Nenashev~A.:
Gysin maps in Balmer-Witt theory.
J. Pure Appl. Algebra 211, 203---221 (2007)

\bibitem{Ne09a}
Nenashev~A.:
Projective push-forwards in the Witt theory of algebraic varieties.
Adv. Math. 220(6), 1923---1944 (2009)

\bibitem{Ne09b}
Nenashev~A.:
On the Witt groups of projective bundles and split quadrics: geometric reasoning.
J. K-Theory 3(3), 533---546 (2009)

\bibitem{NZ06}
Nenashev~A., Zainoulline~K.:
Oriented cohomology and motivic decompositions of relative cellular spaces.
J. Pure Appl. Algebra 205(2), 323---340 (2006)

\bibitem{PS03}
Panin~I.:
Oriented cohomology theories of algebraic varieties.
K-Theory 30, 265---314 (2003)

\bibitem{PS04}
Panin~I.:
Riemann-Roch theorems for oriented cohomology.
Axiomatic, enriched and motivic homotopy theory, 261---333 (2004)

\bibitem{PS09}
Panin~I.:
Oriented Cohomology Theories of Algebraic Varieties II.
Homology, Homotopy and Applications 11(1), 349---405 (2009)

\bibitem{PW10a}
Panin~I., Walter~C.:
Quaternionic Grassmannians and Pontryagin classes in algebraic geometry.
arXiv:1011.0649, http://arxiv.org/abs/1011.0649

\bibitem{PW10c}
Panin~I., Walter~C.:
On the motivic commutative spectrum BO.
arXiv:1011.0650, http://arxiv.org/abs/1011.0650

\bibitem{PW10b}
Panin~I., Walter~C.:
On the algebraic cobordism spectra MSL and MSp.
arXiv:1011.0651, http://arxiv.org/abs/1011.0651

\bibitem{R05}
Riou~J.:
Dualit{\'e} de Spanier--Whitehead en g{\'e}om{\'e}trie alg{\'e}brique.
Comptes Rendus Mathematique 340(6), 431---436 (2005)

\bibitem{Sch10}
Schlichting~M.:
Hermitian K-theory of exact categories.
J. K-theory 5(1), 105---165 (2010)

\bibitem{S07a}
Smirnov~A.: 
Orientations and transfers in cohomology of algebraic varieties.
St. Petersburg Math. J. 18(2), 305---346 (2007)

\bibitem{S07b}
Smirnov~A.: 
Riemann–Roch theorem for operations in cohomology of algebraic varieties.
St. Petersburg Math. J., 18(5), 837---856 (2007)

\bibitem{V98}
Voevodsky~V.:
$\A^1$-homotopy theory.
Doc. Math., Extra Vol. I, 579---604 (1998)

\bibitem{Wa03}
Walter~C.:
Grothendieck-Witt groups of projective bundles.
K-theory Preprint Archives. http://www.math.uiuc.edu/K-theory/0644 (2003) Accessed 9 December 2014

\bibitem{We89}
Weibel~C.:
Homotopy Algebraic $K$-theory.
Contemporary Math. 83, 461---488 (1989)

\bibitem{Zib11}
Zibrowius~M.:
Witt Groups of Complex Cellular Varieties.
Doc. Math. 16, 465---511 (2011)



\end{thebibliography}
\end{document}